\documentclass{amsart}
\usepackage{amssymb,amsmath,amsthm}
\usepackage{tikz}
\usepackage{float}
\usepackage{url}
\usepackage{pgfplots}
\pgfplotsset{compat=1.11}
\usetikzlibrary{calc}
\usepackage{graphicx}

\newtheorem{theorem}{Theorem}[section]
\newtheorem*{theorem*}{Theorem}

\newtheorem{corollary}{Corollary}[theorem]
\newtheorem{lemma}[theorem]{Lemma}

\theoremstyle{remark}
\newtheorem{remark}{Remark}[section]
\newcommand{\R}{\mathbb{R}}
\newcommand{\Z}{\mathbb{Z}}
\newcommand{\T}{\mathbb{T}^2}
\newcommand{\Td}{\mathbb{T}^d}
\newcommand{\dd}{\mathcal{D}}

\definecolor{orange}{gray}{0.4}
\title[Lattice Point Discrepancy Estimates]{Higher Moments for Lattice Point Discrepancy of Convex Domains and Annuli}
\author{Xiaorun Wu\vspace{-1em}}
\email{xiaorunw@princeton.edu}

\begin{document}
\maketitle 

\begin{abstract}
    Given a domain $\Omega \subseteq \R^2$, let $\dd(\Omega,x,R)$ be the number of lattice points from $\Z^2$ in $R\Omega -x$, for $R \ge 1$ and $x\in \T$, minus the area of $R\Omega$: $$\dd(\Omega,x,R) = \# \{ (j,k) \in \mathbb{Z}^2 :(j-x_1,k-x_2) \in R\Omega \} - R^2|\Omega|.$$ We call $\int_{\mathbb{T}^2}|\dd(\Omega,x,R)|^pdx$ the $p$-th moment of the discrepancy function $\dd$. In 2014, Huxley showed that for convex domains with sufficiently smooth boundary, the fourth moment of $\dd$ is bounded by $\mathcal{O}(R^2\log R)$, and in 2019, Colzani, Gariboldi, and Gigante extended this result to higher dimensions.
    
    In this paper, our contribution is twofold: first, we present a simple direct proof of Huxley's 2014 result; second, we establish new estimates for the $p$-th moments of lattice point discrepancy of annuli of radius $R$, and any fixed thickness $0<t<1$ for $p\ge 2.$
\end{abstract}

\section{Introduction and Motivation}

\subsection{Background}
Define $N: [1,\infty)\rightarrow \R$ by
\begin{equation}\label{circledis}
N(R) = \#\{(n_1,n_2) \in \Z^2: n_1^2+n_2^2 \leq R^2\}, 
\end{equation}
that is, $N(R)$ is the number of lattice points from $\Z^2$ inside the disk of radius $R$ centered at the origin. In \cite{Gauss}, Gauss gives the naive estimate $N(R) = \pi R^2+\mathcal{O}(R)$. The proof follows from identifying each lattice point with the square of side length 1 which has the lattice point as its center, see Figure \ref{lattice} below, and noting that the collection of squares contains a disk of radius $R-\sqrt{2}/{2}$ and is contained in a disk of radius $R+\sqrt{2}/{2}$. Hence 
$$
\pi (R-\sqrt{2}/2)^2 \leq N(R)\leq \pi (R
    +\sqrt{2}/2)^2,
$$
which implies that $|N(R)-\pi R^2|\leq 2\sqrt{2}\pi R$.

\begin{figure}[h!]
\centering
    \includegraphics[width = 0.4\textwidth]{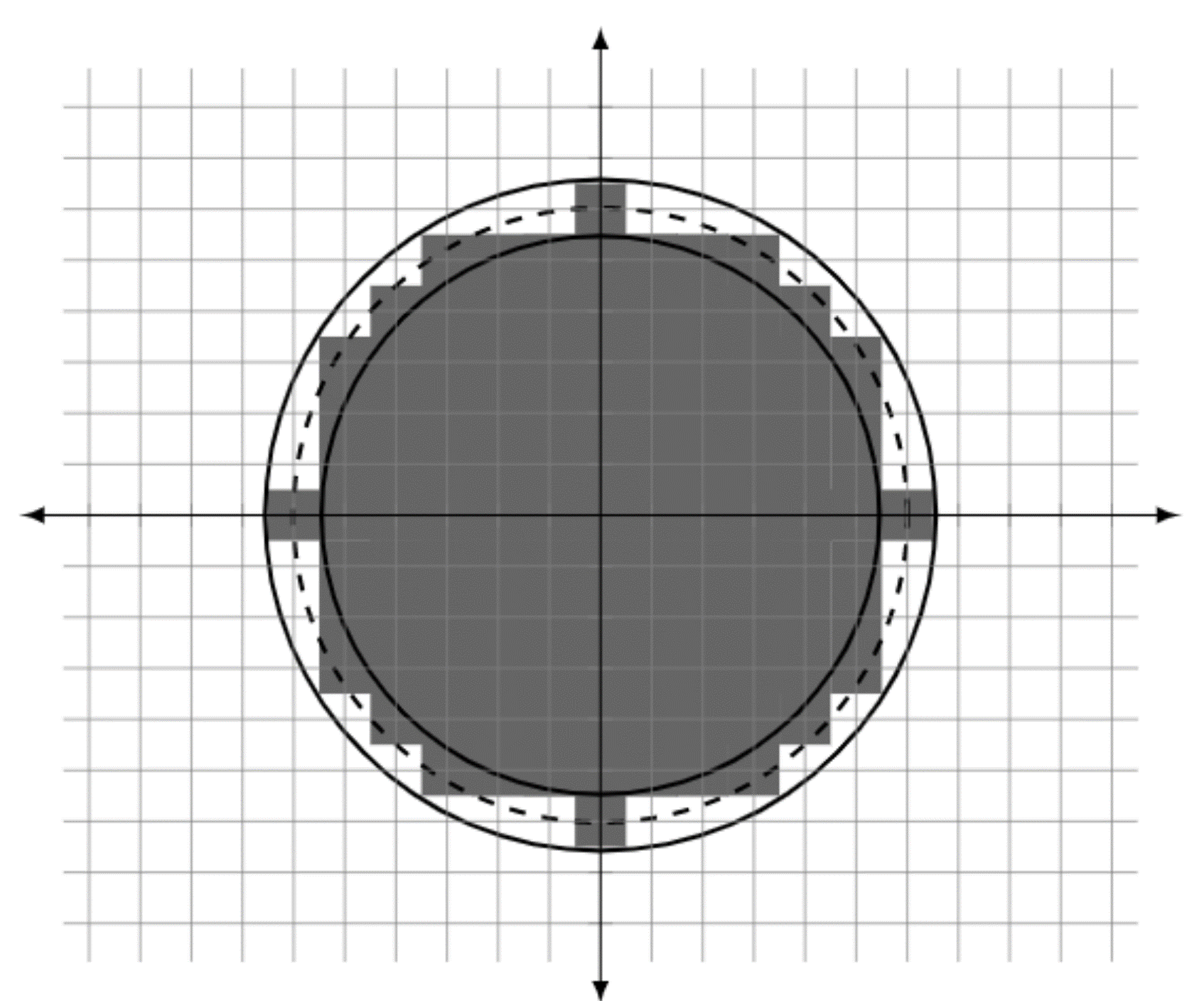}
\caption{An illustration of Gauss's proof.}\label{lattice}
\end{figure}
It is conjectured that
$N(R) = \pi R^2+\mathcal{O}(R^{\frac{1}{2}+\varepsilon})$ for any fixed
$\epsilon > 0$, which is known as the Gauss circle problem. 
There is some empirical evidence that this conjecture is plausible. For example, in Figure \ref{ddd} we plot $(N(R)-\pi R^2)/\sqrt{R}$ for 100,000 values of $R\in[10^6,10^7]$, which seems to suggest that $N(R) = \pi R^2+\mathcal{O}(R^{1/2+\varepsilon})$. 

\begin{figure}[h!]
  \includegraphics[width=0.6\linewidth]{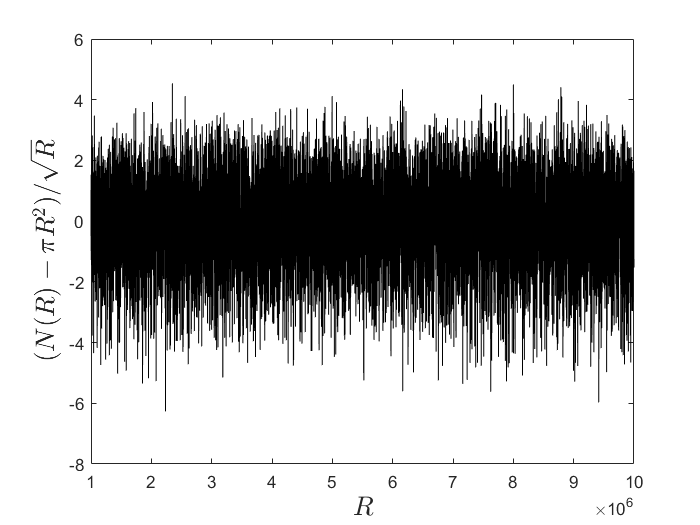}
  \centering
\caption{The function $(N(R)-\pi R^2)/\sqrt{R}$ for 100,000 values of $R\in[10^6,10^7]$.}\label{ddd} 
\end{figure}

Various attempts have been made to bound the discrepancy function $|N(R)-\pi R^2|$. Using techniques from Fourier analysis, Voronoi \cite{Voronoi}, Sierpiński \cite{Sierpinski}, van der Corput \cite{van} improve the naive estimate $|N(R) - \pi R^2| =\mathcal{O}(R)$ to $\mathcal{O}(R^{2/3})$ (see for example Stein and Shakarchi's  \cite{Stein2}). The current best bound known due to Huxley \cite{Huxley2} is $|N(R) - \pi R^2| = \mathcal{O}(R^{131/208})$. On the other side, there are lower bounds. For example,in \cite{Hardy} Hardy established that the discrepancy $|N(R) - \pi R^2|$ cannot be $o\left(R^{{1/2}}(\log R)^{{1/4}}\right)$. Since approaching the problem directly has not led to a solution to the conjecture, many indirect methods of studying the discrepancy $N(R)-\pi R^2$ have emerged.

\subsection{Moment estimation}
One approach to understand the discrepancy function is to study its distribution over shifts $x\in\T$ of the original domain. Let $\Omega\subseteq\R^2$ and $\chi_\Omega$ be its indicator function, where $\T = [0,1]\times [0,1]$. Assume $x = (x_1, x_2)\in\T$, and let $R\Omega-x := \{y\in \R^2: (y+x)/R\in\Omega \}.$ We define $\dd(\Omega,\cdot,\cdot): \T\times[1,\infty)\rightarrow \R$ by
\begin{equation}\label{discre}
\dd(\Omega,x,R) = \sum_{k\in\Z^2}\chi_{R\Omega-x}(k)-R^2|\Omega|,
\end{equation}
and define the $L^p$ norm of the discrepancy function by 
\begin{equation}\label{lpnorm}
 \|\dd(\Omega,x,R)\|_{L^p} := \left(\int_{\mathbb{T}^2}|\dd(\Omega,x,R)|^pdx\right)^{1/p}, \end{equation} 
 for $p \ge 2.$ 
We call the $p$-th moment of $\dd$ 
\begin{equation}\label{moment}
 \|\dd(\Omega,x,R)\|_{L^p}^p := \int_{\mathbb{T}^2}|\dd(\Omega,x,R)|^pdx, 
\end{equation}
for $p \ge 2.$ Kendall showed in \cite{Kendall} that if $\Omega$ is a bounded convex domain whose boundary $\partial \Omega$ is smooth and has nowhere vanishing Gaussian Curvature, then the second moment $\|\dd(\Omega,x,R)\|_{L^2}^2$ is $\mathcal{O}(R)$. Indeed, there is a concise proof of this result using Hardy's identity (see pg. 380-381 of \cite{Stein2}) and Parseval's identity, see for example \cite{Huxley}.

\subsection{Higher moment estimation} It is natural to ask if this analysis can be extended to higher moments. Bounding higher moments is useful for studying the probabilistic distribution of $\dd(\Omega,x,R)$, and could provide insights into the original Gauss circle problem; by the method of moments, if we were able to bound a sequence of higher moments, then it might be possible to obtain a bound for $\|\dd(\Omega,x,R)\|_{L^\infty}$, which is equivalent to bounding $\dd(\Omega,x,R)$ in the almost everywhere sense.

Brandolini, Colzani, Gigante, and Travaglini showed in Theorem 5 of \cite{Brandolini} that for every $p \ge 1$, there exists an constant $C$ such that for every $R > 2$, we have $||\dd(\Omega,x,R)||_{L^p} \ge CR^{1/2}$. It is yet to be shown that the $L^p$ norm is $\mathcal{O}(R^{1/2})$ for $p\ge2$. Again empirically, it is plausible that $\|\dd(\Omega,x,R)\|_{L^p} = \mathcal{O}(R^{1/2})$ for fixed $p \ge 2$, because for different $x\in\T$, $\dd(\Omega,x,R)$ seems to act like uncorrelated random variables. For example, we plotted $\dd(\Omega,x,R)/\sqrt{R}$ against $R$ for various $x\in\T$ in Figure \ref{dd3} below.

\begin{figure}[h!]
\minipage{0.33\textwidth}
  \includegraphics[width=\linewidth]{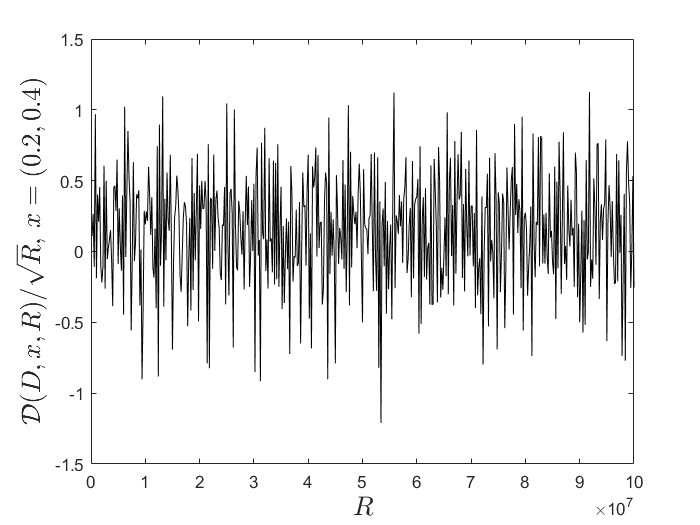}
\endminipage\hfill
\minipage{0.33\textwidth}
  \includegraphics[width=\linewidth]{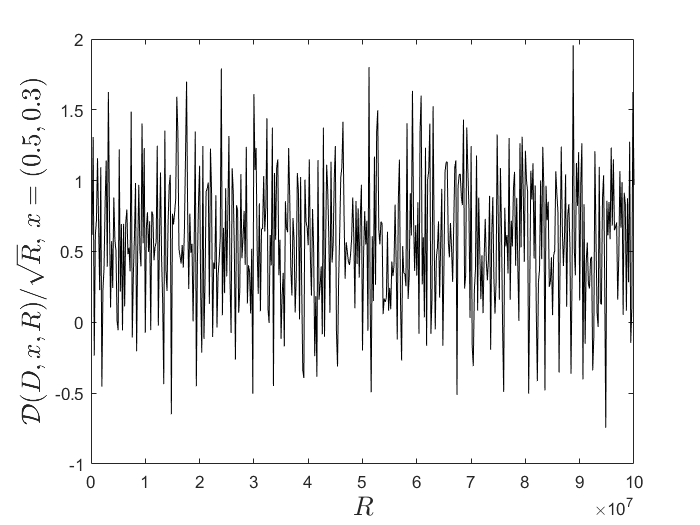}
\endminipage
\minipage{0.33\textwidth}
  \includegraphics[width=\linewidth]{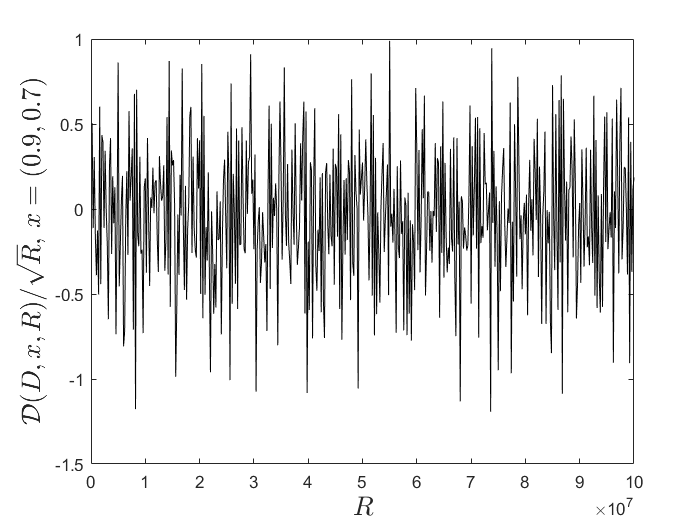}
\endminipage
\caption{The function $\dd(D, x,R)/\sqrt{R}$ for $5,000$ values of $R\in [10^6,10^7]$ and shifts $x= (0.2,0.4),\; (0.5,0.3),\; (0.9,0.7)$ is plotted left, middle, right, respectively. Here $D$ denote unit disk.}\label{dd3}
\end{figure}

In \cite{Huxley}, Huxley established the following estimate for the fourth moment $$\|\dd(\Omega,x,R)\|_{L^4}^4=\mathcal{O}(R^2\log R).$$ Colzani, Gariboldi, and Gigante in \cite{Colzani2} and \cite{Colzani3} have expanded upon Huxley's result and have generalized to higher dimensions. In this paper, we present a simple proof of Huxley's result, and prove new results about annuli, which serve as a potential starting point for further research.

\subsection{Motivation for considering thin annuli}
The discrepancy function of the thin annuli has been extensively studied. For example, Sinai proved in \cite{Sinai} that the number of integer points inside a thin annulus of fixed area $\lambda$, of random shape and sufficiently-large radius $R$, with a suitable definition of randomness, converges in distribution to a Poisson random variable with parameter $\lambda$. One way to study a probabilistic distribution is the method of moments. In this paper, we only consider random shifts, where Sinai's result about the limiting distribution is not known to hold.

We define the annulus $A(R,t)\subseteq\R^2$, of radius $R$ and ring thickness $2t$ by:
$$
A(R,t) = \{y\in \R^2: R-t\leq|y|\leq R+t\}.
$$


\noindent In additional to motivation from Sinai \cite{Sinai}, the discrepancy function of annulus is worth-studying, as it might benefit the study of arithmetic function $r_2(k):= \#\{(n_1,n_2)\in\Z^2: n_1^2+n_2^2 =k\}$: a sufficiently thin annulus might provide extra information about the local behavior of $r_2(k)$ for sufficiently large $k$.
We define the discrepancy function  $\dd(A,\cdot,\cdot,\cdot): \T\times[1,\infty)\times(0,1)\rightarrow \R$ by 
\begin{equation}\label{anndis}
\dd (A,x,R,t) = \sum_{k\in\Z^2}\chi_{A(R,t)-x}(k)-|A(R,t)|.
\end{equation}
The second moment of the discrepancy function is extensively studied. Parnovski and Sidorova in \cite{Parnovski} show that if $t\rightarrow 0$ as $R\rightarrow\infty$, then there is $c>0$ such that for all $R$ large enough, $\|\dd (A,x,R,t)\|_{L^2}\leq cR^{1/2}t^{1/2}$. (See \cite{Parnovski}, pp. 310).
 Colzani, Gariboldi, and Gigante further improve this result to annuli of any dimensions $d$: in Theorem 1 of \cite{Colzani1} they show that for every $\alpha\in\R$ with $\alpha > (d-1)/(d+1)$, there exists $0<\beta < 1$ and $C > 0$ such that for every $1\leq r < +\infty$ and every $0< t \leq r^{-\alpha}$,
$$
\left|\int_{\Td}|\dd(A,x,R,\frac{t}{2})|^2dx-|A(R,\frac{t}{2})|\right|\leq C|A(R,\frac{t}{2})|t^\beta,
$$
in fact, their results hold for more general annuli formed from convex domains.
In this paper, using techniques of Hausdorff Young and interpolation inequalities, we give $p$-th higher moment estimates for all $p \ge 2$ for the thin annuli of radius $R$ and thickness $t$ for any arbitrary $|t|<1$.
\section{Main Results}

\subsection{Notation}
For all subsequent discussions, we write $h(R) \lesssim g(R)$ to denote that
$h(R) \le C g(R)$ for sufficiently large $R$ and an implicit constant $C > 0$. 
We assume that $\Omega \subseteq \R^2$ is a bounded convex domain whose boundary $\partial\Omega$ is smooth and has nowhere vanishing Gaussian curvature, $\chi_\Omega$ denotes its indicator function, and $R\Omega-x :=
\{y\in \R^2: (y+x)/R\in\Omega \}.$ Assume $\dd(\Omega,x,R)$ to be defined as in \eqref{discre},
 and let $\|\dd(\Omega,x,R)\|_{L^p}$ and the $p$-th moment be defined as
 in \eqref{lpnorm} and \eqref{moment} respectively.


\subsection{Main Theorems}
First, we give a simple direct proof of the result about the fourth moment of the discrepancy function, which was initially proved by Huxley in \cite{Huxley}. The techniques used in our proof of this result serve as basis for our investigation of the annuli results.
\begin{theorem}\label{21}
 $\|\dd (\Omega,x,R)\|_{L^4} \lesssim R^{1/2}\log^{1/4}(R).$
\end{theorem}

\noindent We shall defer the proof for Theorem \ref{21} to Section 3.

Next, we present our main results about the annulus. Formally, denote $A(R,t)$ to be the annulus $$A(R,t):\{y\in \R^2: R-t \leq |y| \leq R+t \},$$  and $|A(R,t)|$ to be its area. Define the discrepancy function  $\dd(A,\cdot,\cdot,\cdot): \T\times[1,\infty)\times(0,1)\rightarrow \R$ by $$
\dd (A,x,R,t) := \sum_{k\in\Z^2}\chi_{A(R,t)-x}(k)-|A(R,t)|,$$ 
where $A(R,t)-x:=\{y\in \R^2: (y+x)\in A(R,t) \}.$ Further we define the corresponding $L^p$ norm of the discrepancy function by 
\begin{equation}\label{lpnormannu}
\|\dd (A,x,R,t)\|_{L^p} = \left(\int_{\mathbb{T}^2}|\dd(A,x,R,t)|^pdx\right)^{1/p}, 
 \end{equation} 
 for $p \ge 2$, and $\|\dd(A,x,R,t)\|_{L^p}^p .$ the $p$-th moment of $\dd$.

The following theorem is our second main result, which provides estimates for the $p$-th moments of annuli for $p\ge 2$. 
\begin{theorem}\label{22}
Let $\theta > 0$ be any exponent such that $|\dd(D,x,R)|\lesssim R^\theta$ holds uniformly in $R$ and $x$, where $D$ denotes the unit disk, and assume $p \geq 2$.  Then,  \[
\|\dd(A,x,R,t)\|_{L^p} \lesssim
\begin{cases}
    (Rt)^{\frac{1}{p}}R^{\frac{\theta(p-2)}{p}},          & \text{if } R^{1-2\theta} \geq t\\
    R^{\frac{1}{2}}t^{\frac{4-p}{2p}},            & \text{if } R^{1-2\theta} < t \text{, and } p < 4\\
    R^{\frac{\theta(p-4)+2+\varepsilon}{p}},            & \text{if } R^{1-2\theta} < t \text{, and } p \ge 4,
\end{cases}
    \]
for any fixed $\varepsilon > 0.$
\end{theorem}

 This result serves as a basis for more research on higher moments of annuli, see discussion section in \ref{discussion}. Observe that the estimate depends on $t$ and $\theta$: the first estimate is in general stronger for very thin annuli (e.g. $|t|= o(R^{-1})$), whereas the second is sharper for thicker annuli. We provide two examples in the following corollaries: in the first corollary, we fix $\theta = \frac{2}{3}$ as given by van der Corput (see \cite{van}):
 \begin{corollary}
 Under the hypothesis of Theorem \ref{22}, fix $\theta = \frac{2}{3}$, then 
 \[
\|\dd(A,x,R,t)\|_{L^p} \lesssim
\begin{cases}
    (Rt)^{\frac{1}{p}}R^{\frac{2(p-2)}{3p}},          & \text{if } R^{-1/3} \geq t\\
    R^{\frac{1}{2}}t^{\frac{4-p}{2p}},            & \text{if } R^{-1/3} < t \text{, and } p < 4\\
    R^{\frac{2(p-4)+2+\varepsilon}{3p}} ,           & \text{if } R^{-1/3} < t \text{, and } p \ge 4,
\end{cases}
    \]
for any fixed $\varepsilon > 0.$
 \end{corollary}
 In the second corollary, we fix $t = R^{-1/2}$, and we have the following results: 
\begin{corollary}
 Under the hypothesis of Theorem \ref{22}, fix $t = R^{-1/2}$ and let $p \geq 2$. Then $\|\dd(A,x,R,R^{-1/2})\|_{L^p}\lesssim R^{\frac{\theta(p-2)}{p}+\frac{1}{2p}}$. 
\end{corollary}

Using the estimate of $\theta = \frac{2}{3}$ by van der Corput \cite{van}, we have $\|\dd(A,x,R,t)\|_{L^p}\lesssim  R^{\frac{2}{3}-\frac{5}{6p}}$. So take $p=4$ for example, we have $\|\dd(A,x,R,t)\|_{L^4}\lesssim R^{\frac{11}{24}}$, which is a strict improvement to the $L^4$ moment estimate of the bounded convex domain $\Omega$, where the best known result is $\mathcal{O}(R^{1/2}(\log R)^{1/4})$. 
We shall now proceed proving these theorems.

\section{Proof of Theorem \ref{21}}

\subsection{Technical lemmas}
Suppose that $\varphi : \mathbb{R}^2 \rightarrow
\mathbb{R}$ is a non-negative $C^\infty$ bump function supported on the unit
disc, and set $\varphi_\delta(x) = \delta^{-2} \varphi(x/\delta)$.
For $|\delta| < 1$, define $\dd_\delta(\Omega,\cdot,\cdot) :
\T\times [1,\infty) \rightarrow \mathbb{R}$ by
\begin{equation}\label{discr}
\dd_\delta(\Omega,x,R) = \sum_{k \in \mathbb{Z}^2} \left(\chi_{(R+\delta) \Omega - x} *
\varphi_{|\delta|}(k)\right) - R^2 |\Omega|,
\end{equation}
where $$
(f*g)(k):= \int_{\R^2}f(k-y)g(y)dy.
$$
Let $(a_{\delta,n})_{n\in\Z^2}$ and $(b_{\delta,n})_{n\in\Z^2}$ be the Fourier series of $\dd_\delta(\Omega,x,R)$ and $(\dd_\delta(\Omega,x,R))^2$, respectively, that is,
$$
a_{\delta,n} = \int_{\mathbb{T}^2}
\dd_\delta(\Omega,x,R) e^{-2\pi i n \cdot x} dx, \quad\text{and}\quad b_{\delta,n} =
\int_{\mathbb{T}^2} (\dd_\delta(\Omega,x,R))^2 e^{-2\pi i n \cdot x} dx,
$$ 
for $n \in \mathbb{Z}^2$. Before proving Theorem \ref{21}, we establish two technical lemmas:

\begin{lemma}
Let $0<\delta < 1$. Then,
\begin{equation} \label{fineq}
|\dd(\Omega,x,R)|^p \le |\dd_{-\delta}(\Omega,x,R)|^p + |\dd_\delta(\Omega,x,R)|^p,
\end{equation} 
for all $p\ge 1$.
\end{lemma}

\begin{proof}
We note that \begin{equation} \label{mollify}
\chi_{(R - \delta)\Omega - x} * \varphi_\delta (y) \le \chi_{R \Omega -x }(y)
\le \chi_{(R+\delta) \Omega -x} * \varphi_\delta (y),
\end{equation}
for all $y \in \mathbb{R}^2$. Summing over $k\in\Z^2$ gives
$$
\sum_{k\in\Z^2}\left(\chi_{(R - \delta)\Omega - x} * \varphi_\delta (k)\right) \le \sum_{k\in\Z^2}\chi_{R \Omega -x }(k)
\le \sum_{k\in\Z^2}\left(\chi_{(R+\delta) \Omega -x} * \varphi_\delta (k)\right),
$$
and the result follows from the definition of $\dd$ in \eqref{discr}.
\end{proof}

\begin{lemma}\label{32}
Let $n\in\Z^2\setminus\{\Vec{0}\}$ and set $\delta = R^{-\frac{1}{2}}$.  Then $|b_{\delta, n}| \lesssim R|n|^{-1}$ when $n \le \sqrt{R}$, and $|b_{\delta, n}| \lesssim R^2|n|^{-3}$ when $n > \sqrt{R}$.
\end{lemma}
\begin{proof}
Let $a_{\delta,n}$ and $b_{\delta,n}$ be defined as above. Notice that $$ a_{\delta,\vec{0}} = (R+\delta)^2
|\Omega| - R^2 | \Omega| = 2 \delta R |\Omega| + \delta^2 |\Omega|.
$$
When $n \neq \vec{0}$, we write
$$
a_{\delta,n} = \int_{\mathbb{T}^2} \sum_{k \in \mathbb{Z}^2} \left(\chi_{(R+\delta)
\Omega -x} * \varphi_\delta (k)\right) e^{-2 \pi i n \cdot x} dx = \int_{\mathbb{R}^2}
\left(\chi_{(R+\delta)\Omega} * \varphi_\delta(x)\right) e^{- 2\pi i n \cdot x} dx,
$$
which implies that
$$
a_{\delta,n} = (R+\delta)^2 \widehat{\chi}_\Omega((R+\delta) n)
\widehat{\varphi}(\delta n),
$$
where $\widehat{\chi}_\Omega$ and $\widehat{\varphi}$ denote the Fourier
transforms of $\chi_\Omega$ and $\varphi$, respectively. By the assumptions on
$\Omega$ we have $|\widehat{\chi}_\Omega(\xi)| \lesssim
|\xi|^{-3/2}$, which implies that
\begin{equation} \label{aest}
|a_{\delta,n}| \lesssim R^{1/2} |n|^{-3/2} |\widehat{\varphi}(\delta n)|,
\end{equation}
where $n = (n_1,n_2)$ and $|n| = \sqrt{n_1^2+n_2^2}$. So given the bounds for
$|a_{\delta,\vec{0}}|$ and $|a_{\delta,n}|$ we have $|a_{\delta,
\vec{0}}a_{\delta, n}| \lesssim 2\delta R^{3/2}|n|^{-3/2}|\widehat{\varphi}(\delta
n)|.$ Since $b_{\delta,n}$ might be expressed as convolution of $a_{\delta, j}$, hence for $n \not =
\vec{0} = (0,0)$ we have 
\begin{equation} \label{convolution}
|b_{\delta,n}| \lesssim \sum_{j\in\Z^2} |a_{\delta,j}a_{\delta,n-j}| \lesssim
\frac{\delta R^{3/2}|\widehat{\varphi}(\delta
n)|}{|n|^{3/2}}+R\sum_{\substack{j\in\Z^2, \\ j\neq \vec{0}, j \neq
n}}\frac{|\widehat{\varphi}(\delta j)||\widehat{\varphi}(\delta
(n-j))|}{|n-j|^{3/2}|j|^{3/2}}.
\end{equation}

We proceed by considering two cases: $n\le \sqrt{R}$ and $n > \sqrt{R}$. We note that the choice of $\sqrt{R}$ as the threshold for the two cases is optimal for our argument: see Remark \ref{cutoff}.\\


\noindent\textbf{Case ($n \le \sqrt{R}$):} We use the fact that $|\widehat{\varphi}(\delta j)| = \mathcal{O}(1)$, and $|\widehat{\varphi}(\delta
(n-j))| = \mathcal{O}(1)$. So equation \eqref{convolution} simplifies to:
\begin{equation}\label{cccvolun}
|b_{\delta,n}| \lesssim \delta R^{3/2}|n|^{-3/2}+R\sum_{\substack{j\in\Z^2, \\ j\neq \vec{0}, j \neq n}}\frac{1}{|n-j|^{3/2}|j|^{3/2}}. 
\end{equation}
We break the sum on right hand side of \eqref{cccvolun} into the 3 sums:

\begin{multline}\label{regions}
\sum_{\substack{j\in\Z^2, \\j\neq\vec{0}, j\neq n}} \frac{1}{|n-j|^{3/2}|j|^{3/2}} = \sum_{0 <|j|\leq|n|/2} \frac{1}{|n-j|^{3/2}|j|^{3/2}}\\ +\sum_{0<|n-j|\le|n|/2} \frac{1}{|n-j|^{3/2}|j|^{3/2}} +\sum_{\substack{|j|> |n|/2, |n-j|>|n|/2 }}\frac{1}{|n-j|^{3/2}|j|^{3/2}},
\end{multline}
We estimate the first sum on the right hand side of \eqref{regions} by
\begin{equation}\label{s1}
\sum_{\substack{0<|j|\leq|n|/2}}\frac{1}{|n-j|^{3/2}|j|^{3/2}} \lesssim \frac{1}{|n|^{3/2}} \sum_{\substack{0<|j|\leq|n|/2}}\frac{1}{|j|^{3/2}}  
\lesssim \frac{1}{|n|^{3/2}} \int_{r = 1}^{|n|/2}\frac{1}{r^{3/2}}\cdot rdr
\lesssim \frac{1}{|n|},
\end{equation}
By symmetry (replace $j$ with $n-j$), the second sum on the right hand side of \eqref{regions} is also bounded by $|n|^{-1}$, and we estimate the third sum by
\begin{equation}\label{s3}
\sum_{\substack{|j|>|n|/2, |n-j|>|n|/2}}\frac{1}{|n-j|^{3/2}|j|^{3/2}} \lesssim \sum_{\substack{|j|>|n|/2, \\ j\neq n}}\frac{1}{|j|^3} 
\lesssim \int_{r = |n|/2}^{\infty}\frac{1}{r^{3}}\cdot rdr
\lesssim \frac{1}{|n|}.
\end{equation}
Combining equation \eqref{cccvolun}, \eqref{regions}, \eqref{s1}, and \eqref{s3} we have $|b_{\delta, n}| \lesssim \delta R^{3/2}|n|^{-3/2}+R|n|^{-1}$.\\

\noindent\textbf{Case ($n > \sqrt{R}$):} again we analyze the sum on right hand side of \eqref{convolution} in the 3 sums:
\begin{multline}\label{regions2}
\sum_{\substack{j\in\Z^2, \\j\neq\vec{0}, j\neq n}} \frac{|\widehat{\varphi}(\delta j)||\widehat{\varphi}(\delta
(n-j))|}{|n-j|^{3/2}|j|^{3/2}} = \sum_{0 <|j|\leq|n|/2} \frac{|\widehat{\varphi}(\delta j)||\widehat{\varphi}(\delta
(n-j))|}{|n-j|^{3/2}|j|^{3/2}}\\ +\sum_{0<|n-j|\le|n|/2} \frac{|\widehat{\varphi}(\delta j)||\widehat{\varphi}(\delta
(n-j))|}{|n-j|^{3/2}|j|^{3/2}} +\sum_{\substack{|j|> |n|/2, |n-j|>|n|/2}}\frac{|\widehat{\varphi}(\delta j)||\widehat{\varphi}(\delta
(n-j))|}{|n-j|^{3/2}|j|^{3/2}},
\end{multline}

For the first sum on the right hand side of \eqref{regions2}, we use the fact that $|\widehat{\varphi}(\delta j)| = \mathcal{O}(1)$,  and $|\widehat{\varphi}(\delta (n-j))| = \mathcal{O}(\delta^{-2}(n-j)^{-2})$.
For the second sum on the right hand side of \eqref{regions2}, we use $|\widehat{\varphi}(\delta j)| = \mathcal{O}(\delta^{-2}(n-j)^{-2})$, and $|\widehat{\varphi}(\delta (n-j))| = \mathcal{O}(1)$.
For the third sum on the right hand side of \eqref{regions2}, we use $|\widehat{\varphi}(\delta j)| = \mathcal{O}(\delta^{-1} j^{-1})$, and $|\widehat{\varphi}(\delta (n-j))| = \mathcal{O}(\delta^{-1}(n-j)^{-1})$.

So \eqref{regions2} is simplified to:
\begin{multline}\label{regions3}
\sum_{\substack{j\in\Z^2, \\j\neq\vec{0}, j\neq n}} \frac{|\widehat{\varphi}(\delta j)||\widehat{\varphi}(\delta
(n-j))|}{|n-j|^{3/2}|j|^{3/2}} = \sum_{0 <|j|\leq|n|/2} \frac{\delta^{-2}}{|n-j|^{7/2}|j|^{3/2}}\\ +\sum_{0<|n-j|\le|n|/2} \frac{\delta^{-2}}{|n-j|^{3/2}|j|^{7/2}} +\sum_{\substack{|j|> |n|/2, |n-j|>|n|/2}}\frac{\delta^{-2}}{|n-j|^{5/2}|j|^{5/2}},
\end{multline}
where $\delta$ is a constant to be determined. We estimate the first sum on the right hand side of \eqref{regions3} by
\begin{multline}\label{region21}
\sum_{\substack{0<|j|\leq |n|/2}}\frac{\delta^{-2}}{|n-j|^{7/2}|j|^{3/2}} \lesssim \frac{\delta^{-2}}{|n|^{7/2}} \sum_{\substack{0<|j|\leq |n|/2}}\frac{1}{|j|^{3/2}}
 \lesssim \frac{\delta^{-2}}{|n|^{7/2}} \int_{r = 1}^{|n|/2}\frac{1}{r^{3/2}}\cdot rdr
\\ \lesssim \delta^{-2}|n|^{-3},
\end{multline}
By symmetry (replacing $j$ with $n-j$), the second sum on the right hand side of \eqref{regions3} is also bounded by $\delta^{-2}|n|^{-3}$. We estimate the third sum on the right hand side of \eqref{regions3} by
\begin{equation}\label{region23}
\sum_{\substack{|j|>|n|/2, |n-j|>|n|/2}}\frac{\delta^{-2}}{|n-j|^{5/2}|j|^{5/2}} \lesssim \sum_{\substack{|j|>|n|/2, \\ j\neq n}}\frac{\delta^{-2}}{|j|^5} 
\lesssim \int_{r = |n|/2}^{\infty}\frac{\delta^{-2}}{r^{5}}\cdot rdr
= \delta^{-2}|n|^{-3}.
\end{equation}
Finally, for $n > \sqrt{R}$ sufficiently large, $|\widehat{\varphi}(\delta n)| = \mathcal{O}(\delta^{-2}n^{-2})$, so that
\begin{equation}\label{trail}
\frac{\delta R^{3/2}|\widehat{\varphi}(\delta
n)|}{|n|^{3/2}} \lesssim \frac{\delta^{-1} R^{3/2}}{|n|^{7/2}}    
\end{equation}
Combing equations \eqref{convolution}, \eqref{regions3}, \eqref{region21}, \eqref{region23}, \eqref{trail}, we have $|b_{\delta, n}| \lesssim \delta^{-1} R^{3/2}|n|^{-7/2} + \delta^{-2}R|n|^{-3}$.

If we set $\delta = R^{-\frac{1}{2}}$,  then $|b_{\delta, n}| \lesssim R|n|^{-1}$ when $n \le \sqrt{R}$, and that $|b_{\delta, n}| \lesssim R^2|n|^{-3}$ when $n > \sqrt{R}$. This concludes the proof of lemma \ref{32}.
\end{proof}
\begin{remark}\label{cutoff}
The choice $\sqrt{R}$ is optimal. In fact, if we set the cut-off to be $R^\varepsilon$, i.e. $b_{\delta,n} \lesssim \frac{R}{|n|}$ when $n \le R^\varepsilon$ and $b_{\delta,n} \lesssim \frac{R^2}{|n|^3}$ when $n > R^\varepsilon$, then our argument would give $$\int_{\T}|\dd (\Omega,x,R)|^4dx \lesssim \varepsilon\cdot R^2\log R + R^{(4-4\varepsilon)},$$ where we need $\varepsilon \ge \frac{1}{2}$ for the bump function $\varphi$ to decay sufficiently fast (so that $1 + 2\delta \le 2$). So the choice of $\sqrt{R}$ is optimal.
\end{remark}

\begin{remark}\label{cov0}
For the case where $n = \Vec{0}$, we again note that since $a_{\delta,\vec{0}} = 2 \delta R |\Omega| + \delta^2 |\Omega|$, and $|a_{\delta,n}|
\lesssim R^{1/2} |n|^{-3/2} |\widehat{\varphi}(\delta n)| $, it follows that $$|a_{\delta,
\vec{0}}a_{\delta, n}| \lesssim 2\delta R^{3/2}|n|^{-3/2}|\widehat{\varphi}(\delta
n)|.$$ Now since
\begin{equation} \label{convolution0}
|b_{\delta,\vec{0}}| \lesssim \sum_{j\in\Z^2} |a_{\delta,j}a_{\delta,-j}| \lesssim
\delta^2 R^{2} + R\sum_{\substack{j\in\Z^2, \\ j\neq \vec{0}}}\frac{|\widehat{\varphi}(\delta j)||\widehat{\varphi}(-\delta
j)|}{|j|^{3}},
\end{equation}
using the fact $|\widehat{\varphi}(\delta j)| = |\widehat{\varphi}(\delta j)| = \mathcal{O}(1)$ gives
$$
|b_{\delta, \vec{0}}|\lesssim \delta^2R^2+R\cdot\int_{r = 1}^\infty \frac{1}{r^3}rdr\lesssim\delta^2R^2+R,
$$
and setting $\delta = R^{-\frac{1}{2}}$ gives $|b_{\delta, \vec{0}}|\lesssim R$.
\end{remark}

We now give the proof of Theorem 2.1.
\subsection{Proof of Theorem 2.1}
\begin{proof}
By Parseval's identity, we have $$\int_{\T}(\dd_{\delta}(\Omega,x,R))^4dx = \sum_{n\in\Z^2} |b_{\delta, n}|^2. $$
Since by Remark \ref{cov0} we have $|b_{\delta, \vec{0}}| \lesssim R$, hence 
\begin{align*}
\int_{\T}|\dd_{\delta}(\Omega,x,R)|^4dx &\lesssim R^2+ \sum_{1\leq |n|<\sqrt{R}}\frac{R^2}{|n|^2}+ \sum_{|n|\geq \sqrt{R}}\frac{R^{4}}{|n|^6} \\ 
&\lesssim R^2 + R^2\int_{r=1}^{\sqrt{R}}r^{-2} rdr + R^4\int_{r=\sqrt{R}}^{\infty}r^{-6} rdr \lesssim R^2\log R,
\end{align*}
where in the last inequality we used the fact that $|n| \geq \sqrt{R}.$ A similar argument shows the same result for $\dd_{-\delta}$. Therefore by Lemma 2.1 we conclude that
$$
\int_{\mathbb{T}^2} |\dd(\Omega,x,R)|^4 dx \lesssim \int_{\mathbb{T}^2} \left(|\dd_{-\delta}(\Omega,x,R)|^4
+ |\dd_{\delta}(\Omega,x,R)|^4\right) dx \lesssim R^2 \log R,
$$
which completes the proof.
\end{proof}

\begin{remark}\label{better}
We note that Huxley in \cite{Huxley} has $$ b_n = \mathcal{O} \left(\frac{R}{|n|^2}\log(R|n|)\right)+\mathcal{O}\left(\frac{R^{\theta+\frac{1}{2}}}{|n|^{\frac{3}{2}}}\sqrt{\log(R|n|)}\right),$$ 
when $|n| > \sqrt{R}$, where $\theta > 0$ is a constant such that $\|\dd(\Omega,x,R)\|_{L^\infty} \lesssim R^\theta$ holds uniformly in $R$ provided that $R$ is sufficiently large. So from the above proof one can see that our bound for $b_{\delta,n}$ is stronger, which may be useful in certain situation. We would leave further discussion in section 4.
\end{remark}

\begin{remark}
We note further that in case $\Omega \in\R^2$, Gariboldi in \cite{Gariboldi}, and Colzani, Gariboldi, and Gigante in \cite{Colzani3} have shown the following:
\begin{theorem*}
$\|\dd(\Omega,x,R)\|_{L^p}\lesssim R^{1/2}$ for $2 \leq p < 4$.
\end{theorem*}
This result can be proved with Hausdorff-Young inequality. For details about the proof check Gariboldi's \cite{Gariboldi}, or Colzani, Gariboldi, and Gigante's \cite{Colzani3}.
\end{remark}


\section{Proof of Theorem 2.2}
\subsection{Technical lemmas}
Recall that $$A(R,t):=\{y\in \R^2: R-t \le |y| \le R+t\},$$ with $|t| < 1$ and $R\ge 2$. Denote $|A(R,t)|$ its area. Define the discrepancy function  $\dd(A,\cdot,\cdot,\cdot): \T\times[1,\infty)\times(0,1)\rightarrow \R$ by $$
\dd (A,x,R,t) = \sum_{k\in\Z^2}\chi_{A(R,t)-x}(k)-|A(R,t)|.$$
Denote the Fourier coefficients of $\dd$ by
$$
c_n = \int_{\mathbb{T}^2} \dd (A,x,R,t) e^{-2\pi i n \cdot x} dx,
$$
for $n \in \mathbb{Z}^2$.  
Let $\widehat{\chi}_{A(R,t)}(\xi)$ be the Fourier transform of the indicator function ${\chi}_{A(R,t)}$. The following lemma appears in many places in literature (for example see \cite{Brandolini}); we give a proof for completeness:
\begin{lemma}
Suppose $|t| > 0, R\ge 1$. Then
$$
\widehat{\chi}_{A(R,t)}(\xi) = \frac{2}{\pi} R^{1/2} |\xi|^{-3/2} \sin
(-2\pi R|\xi| + 3\pi/4) \sin(2\pi t|\xi|) + \mathcal{O}(R^{-1/2} t |\xi|^{-3/2}).
$$
\end{lemma}
\begin{proof}
We have
$$
\widehat{\chi}_{A(R,t)}(\xi) = \frac{R+t}{|\xi|} J_1(2 \pi (R+t)|\xi|) -
\frac{R-t}{|\xi|} J_1(2\pi(R-t)|\xi|).
$$
Recall that
$$
J_1(\rho) = \sqrt{\frac{2}{\pi \rho}} \cos(\rho - 3\pi/4) +
\mathcal{O}(\rho^{-3/2}).
$$
It follows that
$$
\frac{t}{|\xi|} \left( J_1(2\pi (R+t) |\xi|) + J_1(2\pi (R-t) |\xi|) \right) =
\mathcal{O}(t R^{-1/2} |\xi|^{-3/2}).
$$
Thus,
$$
\widehat{\chi}_{A(R,t)}(\xi) = \frac{R}{|\xi|} \left( J_1(2 \pi (R+t)|\xi|) -
J_1(2\pi(R-t)|\xi|) \right) + \mathcal{O}(t R^{-1/2} |\xi|^{-3/2}).
$$
We now consider
$$
\psi(R+t) := J_1(2 \pi (R+t)|\xi|) - \frac{1}{\pi} |\xi|^{-1/2}
(R+t)^{-1/2}\cos(2 \pi (R+t) |\xi|-3\pi/4).
$$
Recall that 
$$
J'_1(\rho) = -\sqrt{\frac{2}{\pi}} \rho^{-1/2} \sin(\rho - 3\pi/4) +
\mathcal{O}(\rho^{-3/2}).
$$
Thus, when $R \le a \le R+t$ we have
$$
\psi'(a) = \mathcal{O}(R^{-3/2} |\xi|^{-3/2}).
$$
Thus, a Taylor expansion of $\psi$ gives
$$
\psi(R+t) = J_1(2\pi R |\xi|) - \frac{2}{\pi} |\xi|^{-3/2} R^{-1/2} \cos(2\pi R
|\xi| -3\pi/4) + \mathcal{O}(t R^{-3/2} |\xi|^{-3/2}).
$$
We conclude that
$$
\psi(R+t) - \psi(R-t) = \mathcal{O}(t R^{-3/2} |\xi|^{-3/2}).
$$
Thus
\begin{multline*}
\widehat{\chi}_{A(R,t)}(\xi) = \frac{1}{\pi} R^{-1/2}|\xi|^{-3/2} \left(
\cos(2\pi(R+t)|\xi| -3\pi/4) - \cos(2\pi(R-t)|\xi| -3\pi/4) \right)\\ +
\mathcal{O}(t R^{-1/2} |\xi|^{-3/2}).
\end{multline*}
Since
\begin{multline*}
\cos(2\pi(R+t)|\xi| -3\pi/4) - \cos(2\pi(R-t)|\xi| -3\pi/4) \\
= -2 \sin(2\pi R |\xi| - 3\pi/4) \sin(2\pi t |\xi|),
\end{multline*}
it follows that
$$
\widehat{\chi}_{A(R,t)}(\xi) = \frac{2}{\pi} R^{1/2} |\xi|^{-3/2} \sin (-2\pi R|\xi| +
3\pi/4) \sin(2\pi t|\xi|) + \mathcal{O}(R^{-1/2} t |\xi|^{-3/2}),
$$
as was to be shown.
\end{proof}
We also require the following lemma to prove Theorem \ref{22}:

\begin{lemma}\label{annulemma}
Suppose that $2\le p<4$. Then, $$\|\dd (A,x,R,t)\|_{L^p}\lesssim R^{1/2}\left( t^{\frac{p}{8-2p}} \right).$$ 
\end{lemma}
\begin{proof}
Denote the fourier coefficients of $\dd(A,x,R,t)$ as
$$
c_n = \int_{\mathbb{T}^2}\dd(A,x,R,t)e^{- 2\pi i n \cdot x} dx,
$$
so that
$$
c_{n} = \int_{\mathbb{T}^2} \left(\sum_{k \in \mathbb{Z}^2} \chi_{A(R,t)-x}(k)\right) e^{-2 \pi i n \cdot x} dx = \int_{\mathbb{R}^2}
\chi_{A(R,t)}(x) e^{- 2\pi i n \cdot x} dx,
$$
which implies that
$$
c_{n} = \widehat{\chi}_{A(R,t)}(n)
,
$$
where $\widehat{\chi}_{A(R,t)}$ denotes the Fourier transform of $\chi_{A(R,t)}$. Lemma 4.1 implies that $|\widehat{\chi}_{A(R,t)}(\xi)| \lesssim
R^{1/2}|\xi|^{-3/2}\sin(2\pi t|\xi|)$, so that
\begin{equation} \label{annuliest}
c_n = \mathcal{O}(R^{1/2} |n|^{-3/2}\sin(2\pi t|n|)),
\end{equation}
for $n\in\Z^2\setminus\{\Vec{0}\}$. Again we note that $c_{\Vec{0}} = 0$.
The Hausdorff-Young inequality states that for $\dd(A,\cdot,\cdot,\cdot): \T\times[1,\infty)\times(0,1)\rightarrow \R$ with Fourier coefficients $(c_n)_{n\in\Z^2}$, we have 
$$
\left( \int_{\mathbb{T}^2} |\dd (A,x,R,t)|^p dx \right)^{1/p} \lesssim \left( \sum_{n \in
\mathbb{Z}^2} |c_n|^q \right)^{1/q},
$$
when $2 \le p \le \infty$ and $1/p + 1/q = 1$. Fix $2 \ge \varepsilon > 0$, and set
$p = 4-\varepsilon$ such that $q = (4-\varepsilon)/(3-\varepsilon)$. It follows
from the Hausdorff-Young inequality and \eqref{annuliest} that
\begin{equation} \label{estt}
\left( \int_{\mathbb{T}^2} \left| \dd (A,x,R,t)\right|^{4-\varepsilon} dx \right)^\frac{1}{4-\varepsilon}
\lesssim \left( \sum_{n \in \mathbb{Z}^2, n \not =\vec{0}}
c_n^{\frac{4-\varepsilon}{3-\varepsilon}}
\right)^{\frac{3-\varepsilon}{4-\varepsilon}},
\end{equation}
 So now consider the summation $$ \sum_{n \in \Z^2\setminus\{\Vec{0}\}}
c_n^{\frac{4-\varepsilon}{3-\varepsilon}}
 =  \sum_{0<|n| \le 1/t }
c_n^{\frac{4-\varepsilon}{3-\varepsilon}}
 +  \sum_{|n| > 1/t }
c_n^{\frac{4-\varepsilon}{3-\varepsilon}}
$$ 
So when $|n|\le 1/t$, we have that $\sin(2\pi t|n|) < 2\pi t|n|$, so that $$c_n = \mathcal{O}\left(R^{1/2} |n|^{-3/2}\sin(2\pi t|n|)\right) = \mathcal{O}(R^{1/2}|n|^{-1/2}t).$$ 
Hence we have 
\begin{align*}
    \sum_{n \in \Z^2\setminus\{\Vec{0}\}}
c_n^{\frac{4-\varepsilon}{3-\varepsilon}} &= \sum_{0<|n| \le 1/t }c_n^{\frac{4-\varepsilon}{3-\varepsilon}}
 +  \sum_{|n| > 1/t }
c_n^{\frac{4-\varepsilon}{3-\varepsilon}} \\
&\lesssim R^{\frac{4-\varepsilon}{6-2\varepsilon}}\left(t^{\frac{4-\varepsilon}{3-\varepsilon}}\sum_{0<|n|\le 1/t}n^{\frac{\varepsilon-4}{6-2\varepsilon}} + \sum_{|n| > 1/t} n^{\frac{3\varepsilon-12}{6-2\varepsilon}} \right).
\end{align*}
Now since
$$t^{\frac{4-\varepsilon}{3-\varepsilon}}\sum_{0 <|n| \le 1/t}n^{\frac{\varepsilon-4}{6-2\varepsilon}} + \sum_{|n| > 1/t} n^{\frac{3\varepsilon-12}{6-2\varepsilon}} \lesssim t^{\frac{4-\varepsilon}{3-\varepsilon}}\int_{r = 1}^{1/t} r^{\frac{\varepsilon-4}{6-2\varepsilon}}\cdot rdr + \int_{r = 1/t}^{\infty} r^{\frac{3\varepsilon-12}{6-2\varepsilon}}\cdot rdr \lesssim  t^{\frac{\epsilon}{6-2\epsilon}}.$$

Hence we have
$$
\left( \sum_{n \in \Z^2\setminus\{\Vec{0}\}}
c_n^{\frac{4-\varepsilon}{3-\varepsilon}}
\right)^{\frac{3-\varepsilon}{4-\varepsilon}} \lesssim R^{1/2}\left( t^{\frac{\varepsilon}{8-2\varepsilon}} \right),
$$
which concludes the proof for lemma \ref{annulemma}.
\end{proof}
\begin{remark}
Note in particular, setting $t = R^{-1/2}$ gives us
$$
\left\| \dd (A,x,R,t) \right\|_{L^{4-\varepsilon}(\mathbb{T}^2)}=\left(
\int_{\mathbb{T}^2} \left| \dd (A,x,R,t) \right|^{4-\varepsilon} dx \right)^{\frac{1}{4-\varepsilon}}
\lesssim R^{\frac{8-3\varepsilon}{16-4\varepsilon}},
$$
where we emphasize that the implicit constant depends on $\varepsilon > 0$.
\end{remark}
We now give the proof of Theorem \ref{22}.
\subsection{Proof of Theorem 2.2}
\begin{proof}
Let $p \ge 2$ be a fixed constant, and let $2\leq p_0 \leq p$. By monotonicity of integration, we have 
\begin{multline}
\int_{\T}|\dd(A,x,R,t)|^pdx \lesssim \int_{\T}|\dd(A,x,R,t)|^{p_0} (\sup_{x\in\T}|\dd(A,x,R,t)|^{p-p_0})dx \\ \lesssim \|\dd(A,x,R,t)\|_{L^{p_0}}^{p_0}\cdot \|\dd(A,x,R,t)\|_{L^\infty}^{(p-p_0)}.
\end{multline}
Now let $D$ be the unit disk, and let $\theta\in\R$ be any positive constant such that $|\dd(D,x,R)| \lesssim R^\theta$ holds uniformly for all $R$, if $R$ is sufficiently large. From the fact that
$$
\dd(A,x,R,t) = \dd(D,x,R+t) - \dd (D,x,R-t) \lesssim \mathcal{O}(R^\theta),
$$
and the results from lemma \ref{annulemma}, we have
\begin{equation}\label{minimize}
  \|\dd(A,x,R,t)\|_{L^p}^p \lesssim  R^{\frac{p_0}{2}}t^{\frac{4-p_0}{2}}R^{\theta(p-p_0)}= (R^{\theta p}t^2)\cdot(R^{\frac{1}{2}-\theta}t^{-\frac{1}{2}})^{p_0}.
\end{equation}
So now we want to minimize \eqref{minimize}, and we split our discussion into two cases:

\textbf{Case 1:} when $R^{\frac{1}{2}-\theta}t^{-\frac{1}{2}}\ge 1$, or equivalently $t\leq R^{1-2\theta}$, we shall set $p_0 $ as small as possible, so that $p_0 =2$. Hence substitute back to \eqref{minimize} and take $1/p$-th power we have
$$
\|\dd(A,x,R,t)\|_{L^p}\lesssim (Rt)^{\frac{1}{p}}R^{\frac{\theta(p-2)}{p}}.
$$

\textbf{Case 2:} when $R^{\frac{1}{2}-\theta}t^{-\frac{1}{2}}< 1$, or equivalently $t > R^{1-2\theta}$, we shall set $p_0$ as large as possible, so we further split into two sub cases:
If $p < 4$, then setting $p_0 = p$ and taking $1/p$-th power on both sides yields 
$$
\|\dd(A,x,R,t)\|_{L^p}\lesssim R^{\frac{1}{2}}t^{\frac{(4-p)}{2p}}.
$$
If $p \ge 4$, then we note lemma \ref{annulemma} holds only for moments less than 4, so we set $p_0 = 4-\varepsilon_0$ for some fixed $\varepsilon_0 > 0$. So \eqref{minimize} becomes
$$
\|\dd(A,x,R,t)\|_{L^p}^p\lesssim R^{\theta(p-4)+2}\cdot R^{\frac{\varepsilon_0}{2}(2\theta -1 +\alpha)},
$$
assuming $t = R^\alpha$ for some $\alpha<0$. Since $t> R^{1-2\theta} $, hence $\alpha + 2\theta - 1 > 0$. Thus setting $\varepsilon = \frac{\varepsilon_0}{2}(\alpha + 2\theta - 1)$ and taking the $1/p$-th power on both sides yields $$
\|\dd(A,x,R,t)\|_{L^p}\lesssim R^{\frac{\theta(p-4)+2+\varepsilon}{p}}\cdot 
$$
This concludes the proof of Theorem 2.2.
\end{proof}

\section{Discussion}
We discuss some limitations and directions for further research.
\subsection{Moments of bounded convex domains.} First, it seems that we cannot get better fourth moment estimation using current bump function and convolution techniques: as we noted in Remark \ref{cutoff}, changing the cutoff does not improve the result on fourth moment further. 

Moreover, despite that our result on $b_{\delta,n}$ is slightly stronger than Huxley's original estimates for $b_n$, it is still poor in estimating higher moments. To demonstrate this point, we take for example the $L^8$ norm: using Hausdorff-Young Inequality, we obtained the following estimates:

\begin{multline*}\label{8thmoement}
    \left(\int_{\T}|\dd_\delta (\Omega,x,R)^2|^4\right)^{1/4} \lesssim \left(\sum_{n\in\Z^2}b_{\delta,n}^{4/3}\right)^{3/4} = \left(\sum_{|n|<\sqrt{R}}b_{\delta,n}^{4/3}+\sum_{|n|\ge\sqrt{R}}b_{\delta,n}^{4/3}\right)^{3/4}\\ \lesssim R^{5/4},
\end{multline*}
so a similar mollification argument $|\dd(\Omega,x,R)|^8 \lesssim |\dd_\delta (\Omega,x,R)|^8+|\dd_{-\delta} (\Omega,x,R)|^8$ gives us $$\|\dd(\Omega,x,R)\|_{L^8} \lesssim R^{5/8}.$$ If we repeat the techniques we used to prove Theorem 2.1, we will have increasingly worse upper bound for higher moments (assuming the summation of powers of $b_{\delta,n}$ still converges): the upper bound approaches $\mathcal{O}(R)$ as the left hand side power tends to infinity. 

One possible research direction is to make use of the cosine term in the original Bessel integral of the Fourier coefficient $a_{\delta,n}$. In fact, from Hardy's identity we have 
$$
a_n = \mathcal{O}(\frac{R}{|n|}|J_1(2\pi|n|R)|),
$$
where $J_1$ is the Bessel function of first kind.  We note that since
$$
J_1(s) = \sqrt{\frac{2}{\pi s}}\cos(s-\frac{3\pi}{4}) +\mathcal{O}(s^{-3/2}),
$$
as $s\rightarrow\infty$ (see Stein \cite{Stein1}), thus in fact
$$
a_n = \frac{R^{1/2}}{|n|^{3/2}}\cos (2\pi |n|R - \frac{3\pi}{4})+ \mathcal{O}(R^{-1/2}|n|^{-5/2}),
$$
so making use of the cosine term in bounding $b_{\delta, n}$ might help us obtain better bounds.

\subsection{Moments of thin annuli}\label{discussion}
We further note that the bounds for the thin annuli could be improved. Our current results do not seem to be sharp, as Sinai \cite{Sinai}, and Colzani, Gariboldi, and Gigante \cite{Colzani1} noted that the discrepancy function of the thin annuli would ideally follow Poisson distribution, so we might be able to improve the current $4th$ moment estimates to $\mathcal{O}(R)$ for $t = R^{-1/2}$.

It is also worth noting that even if the original Gauss circle conjecture were to be true (i.e. $\theta = 1/2+\varepsilon$ for any fixed $\varepsilon > 0$), for sufficiently large $p$ (for example $p > 4$) Theorem \ref{22} above suggests that $\|\dd(A,x,R,t)\|_{L^p}\lesssim R^{1/2+\varepsilon_1}$ for any fixed $\varepsilon_1 >0$. This suggests when $t$ becomes sufficiently large, the cancellation effect from the inner circle of the thicker annulus becomes less prominent as compared to thinner annulus. Thus a possible further research direction is to study what is the critical value for $t$.

\section*{Acknowledgements} Special thanks has to be given to Nicholas Marshall, whose invaluable suggestions have made this paper possible. Also I'd like to thank Princeton University, Department of Mathematics for offering Undergraduate Mathematics Summer Funding to facilitate this research.

\end{document}